\documentclass{article}[11pt]
\usepackage{titling}
\usepackage{ amssymb }
\usepackage{amsfonts}
\usepackage{amsthm}
\usepackage{amsmath}
\usepackage{bbm}
\usepackage{setspace}
\usepackage[normalem]{ulem}
\usepackage{geometry}
\usepackage{float}

\usepackage[utf8]{inputenc}

\def\showauthornotes{1}

\def\showkeys{0}
\def\showdraftbox{0}
\def\showcolorlinks{1}
\def\usemicrotype{1}
\def\showfixme{0}

\usepackage{csquotes}

\usepackage[
backend=biber,
maxcitenames=50,
maxbibnames=99,
style=alphabetic,
]{biblatex}
 \usepackage{titlesec}
\title{Conditional Non-Soficity of p-adic Deligne Extensions:\\on a Theorem of Gohla and Thom}
\author{Michael Chapman, Yotam Dikstein, Alexander Lubotzky}
\date{\today}

\DeclareMathOperator*{\Probability}{\mathbb{P}}
\newcommand{\Pro}[1]{\Probability_{#1}}
\DeclareMathOperator*{\Expectation}{\mathbb{E}}
\newcommand{\Ex}[1]{\Expectation_{#1}}

\usepackage{comment}
\usepackage[l2tabu, orthodox]{nag}

\usepackage{xspace,enumerate}

\usepackage[dvipsnames]{xcolor}

\usepackage[T1]{fontenc}
\usepackage[full]{textcomp}

\usepackage[american]{babel}

\usepackage{mathtools}

\usepackage{amsthm}
\usepackage{amsmath}

\newtheorem{theorem}{Theorem}[section]
\newtheorem*{theorem*}{Theorem}

\newtheorem{claim}[theorem]{Claim}
\newtheorem*{claim*}{Claim}

\newtheorem*{proposition*}{Proposition}

\newtheorem*{lemma*}{Lemma}
\newtheorem{corollary}[theorem]{Corollary}
\newtheorem*{corollary*}{Corollary}
\newtheorem*{conjecture*}{Conjecture}
\newtheorem{fact}[theorem]{Fact}
\newtheorem*{fact*}{Fact}

\newtheorem*{hypothesis*}{Hypothesis}

\theoremstyle{definition}
\newtheorem{definition}[theorem]{Definition}
\newtheorem*{definition*}{Definition}

\newtheorem*{observation*}{Observation}

\theoremstyle{remark}

\newtheorem{remark}[theorem]{Remark}
\newtheorem*{remark*}{Remark}

\usepackage[varg]{pxfonts} % varg - uses nicer g,v,y,w

\usepackage[sc,osf]{mathpazo}
\usepackage{lmodern}

\ifnum\showkeys=1
\usepackage[color]{showkeys}
\fi

\ifnum\showcolorlinks=1
\usepackage[
colorlinks=true,
urlcolor=blue,
linkcolor=blue,
citecolor=OliveGreen,
]{hyperref}
\fi

\ifnum\showcolorlinks=0
\usepackage[
colorlinks=false,
pdfborder={0 0 0}
]{hyperref}
\fi

\usepackage{prettyref}

\newcommand{\savehyperref}[2]{\texorpdfstring{\hyperref[#1]{#2}}{#2}}

\newrefformat{eq}{\savehyperref{#1}{\textup{(\ref*{#1})}}}
\newrefformat{lem}{\savehyperref{#1}{Lemma~\ref*{#1}}}
\newrefformat{def}{\savehyperref{#1}{Definition~\ref*{#1}}}
\newrefformat{thm}{\savehyperref{#1}{Theorem~\ref*{#1}}}
\newrefformat{cor}{\savehyperref{#1}{Corollary~\ref*{#1}}}
\newrefformat{cha}{\savehyperref{#1}{Chapter~\ref*{#1}}}
\newrefformat{sec}{\savehyperref{#1}{Section~\ref*{#1}}}
\newrefformat{app}{\savehyperref{#1}{Appendix~\ref*{#1}}}
\newrefformat{tab}{\savehyperref{#1}{Table~\ref*{#1}}}
\newrefformat{fig}{\savehyperref{#1}{Figure~\ref*{#1}}}
\newrefformat{hyp}{\savehyperref{#1}{Hypothesis~\ref*{#1}}}
\newrefformat{alg}{\savehyperref{#1}{Algorithm~\ref*{#1}}}
\newrefformat{rem}{\savehyperref{#1}{Remark~\ref*{#1}}}
\newrefformat{item}{\savehyperref{#1}{Item~\ref*{#1}}}
\newrefformat{step}{\savehyperref{#1}{step~\ref*{#1}}}
\newrefformat{conj}{\savehyperref{#1}{Conjecture~\ref*{#1}}}
\newrefformat{fact}{\savehyperref{#1}{Fact~\ref*{#1}}}
\newrefformat{prop}{\savehyperref{#1}{Proposition~\ref*{#1}}}
\newrefformat{prob}{\savehyperref{#1}{Problem~\ref*{#1}}}
\newrefformat{claim}{\savehyperref{#1}{Claim~\ref*{#1}}}
\newrefformat{relax}{\savehyperref{#1}{Relaxation~\ref*{#1}}}
\newrefformat{red}{\savehyperref{#1}{Reduction~\ref*{#1}}}
\newrefformat{part}{\savehyperref{#1}{Part~\ref*{#1}}}
\newrefformat{ex}{\savehyperref{#1}{Example~\ref*{#1}}}
\newrefformat{para}{\savehyperref{#1}{Paragraph~\ref*{#1}}}
\newrefformat{ass}{\savehyperref{#1}{Assumption~\ref*{#1}}}
\newrefformat{question}{\savehyperref{#1}{Question~\ref*{#1}}}

\newcommand{\Sref}[1]{\hyperref[#1]{\S\ref*{#1}}}

\usepackage{nicefrac}
\ifnum\usemicrotype=1
\usepackage{microtype}
\fi

\ifnum\showauthornotes=1
\newcommand{\Authornote}[2]{{\sffamily\small\color{red}{[#1: #2]}}}
\newcommand{\Authornotecolored}[3]{{\sffamily\small\color{#1}{[#2: #3]}}}
\newcommand{\Authorcomment}[2]{{\sffamily\small\color{gray}{[#1: #2]}}}
\newcommand{\Authorstartcomment}[1]{\sffamily\small\color{gray}[#1: }

\newcommand{\Authorfnote}[2]{\footnote{\color{red}{#1: #2}}}
\newcommand{\Authorfixme}[1]{\Authornote{#1}{\textbf{??}}}
\newcommand{\Authormarginmark}[1]{\marginpar{\textcolor{red}{\fbox{\Large #1:!}}}}
\else
\newcommand{\Authornote}[2]{}
\newcommand{\Authornotecolored}[3]{}
\newcommand{\Authorcomment}[2]{}
\newcommand{\Authorstartcomment}[1]{}

\newcommand{\Authorfnote}[2]{}
\newcommand{\Authorfixme}[1]{}
\newcommand{\Authormarginmark}[1]{}
\fi

\ifnum\showfixme=0

\fi

\usepackage{boxedminipage}

\newcommand{\set}[1]{\{#1\}}

\newcommand{\iprod}[1]{\langle#1\rangle}

\newcommand{\F}{\mathbb{F}}

\newcommand{\Psymb}{\mathbb{P}}

\DeclareMathOperator*{\ProbOp}{\Psymb}

\renewcommand{\Pr}{\ProbOp}
\def\one{{\mathbf{1}}}

\newcommand{\Prob}[2][]{\Pr_{{#1}}\left[#2\right]} % use by \Prob[x]{event}
 \usepackage{dsfont}
\usepackage{mathrsfs}

\newcommand{\textparen}[1]{\text{(#1)}}

\ifx\because\undefined
\newcommand{\because}[1]{\textparen{because #1}}
\else
\renewcommand{\because}[1]{\textparen{because #1}}
\fi

\newcommand\bdot\bullet

\DeclareMathOperator{\sign}{sign}

\newcommand{\Z}{\mathbb Z}

\newcommand{\cF}{\mathcal F}

\renewcommand{\leq}{\leqslant}

\renewcommand{\geq}{\geqslant}

\ifnum\showdraftbox=1

\else

\fi

\let\epsilon=\varepsilon

\numberwithin{equation}{section}

\newcommand{\MYstore}[2]{%
  \global\expandafter \def \csname MYMEMORY #1 \endcsname{#2}%
}

\newcommand{\MYload}[1]{%
  \csname MYMEMORY #1 \endcsname%
}

\newcommand{\MYnewlabel}[1]{%
  \newcommand\MYcurrentlabel{#1}%
  \MYoldlabel{#1}%
}

\newcommand{\MYdummylabel}[1]{}

\newcommand{\torestate}[1]{%
  \let\MYoldlabel\label%
  \let\label\MYnewlabel%
  #1%
  \MYstore{\MYcurrentlabel}{#1}%
  \let\label\MYoldlabel%
}

\newcommand{\restatetheorem}[1]{%
  \let\MYoldlabel\label
  \let\label\MYdummylabel
  \begin{theorem*}[Restatement of \prettyref{#1}]
    \MYload{#1}
  \end{theorem*}
  \let\label\MYoldlabel
}

\newcommand{\restatelemma}[1]{%
  \let\MYoldlabel\label
  \let\label\MYdummylabel
  \begin{lemma*}[Restatement of \prettyref{#1}]
    \MYload{#1}
  \end{lemma*}
  \let\label\MYoldlabel
}

\newcommand{\restateprop}[1]{%
  \let\MYoldlabel\label
  \let\label\MYdummylabel
  \begin{proposition*}[Restatement of \prettyref{#1}]
    \MYload{#1}
  \end{proposition*}
  \let\label\MYoldlabel
}

\newcommand{\restateclaim}[1]{%
  \let\MYoldlabel\label
  \let\label\MYdummylabel
  \begin{claim*}[Restatement of \prettyref{#1}]
    \MYload{#1}
  \end{claim*}
  \let\label\MYoldlabel
}

\newcommand{\restatecorollary}[1]{%
  \let\MYoldlabel\label
  \let\label\MYdummylabel
  \begin{corollary*}[Restatement of \prettyref{#1}]
    \MYload{#1}
  \end{corollary*}
  \let\label\MYoldlabel
}

\newcommand{\restatefact}[1]{%
  \let\MYoldlabel\label
  \let\label\MYdummylabel
  \begin{fact*}[Restatement of \prettyref{#1}]
    \MYload{#1}
  \end{fact*}
  \let\label\MYoldlabel
}

\newcommand{\restatedefinition}[1]{% % overwrite label command with dummy
\let\MYoldlabel\label 
\let\label\MYdummylabel 
\begin{definition*}[Restatement of \prettyref{#1}] 
    \MYload{#1} 
\end{definition*} 
\let\label\MYoldlabel 
} 

\newcommand{\restate}[1]{%
  \let\MYoldlabel\label
  \let\label\MYdummylabel
  \MYload{#1}
  \let\label\MYoldlabel
}

\newcommand{\eps}{\epsilon}

\let\origparagraph\paragraph
\renewcommand{\paragraph}[1]{\origparagraph{#1.}}

\allowdisplaybreaks

\let\pref=\prettyref

\newcommand{\dir}[1]{\overset{\to}{#1}}
\newcommand{\bproof}[1]{\begin{proof}[Proof of \pref{#1}]}
\newcommand{\eproof}{\end{proof}}
\newcommand{\coboundary}{{\delta}}

\newcommand{\Img}{{\textrm {Im}}}

\begin{document}

\maketitle
\begin{abstract}
   A long standing problem asks whether every group is sofic, i.e., can be separated by almost-homomorphisms to the symmetric group \(\textup{Sym}(n)\). Similar problems have been asked with respect to almost-homomorphisms to the unitary group \(U(n)\), equipped with various norms. One of these problems has been solved, for the first time in \cite{de2020stability}: some central extensions \(\widetilde{\Gamma}\) of arithmetic lattices $\Gamma$ of \(Sp(2g,\mathbb{Q}_p)\) were shown to be non-Frobenius approximated by almost homomorphisms to \(U(n)\). Right after, it was shown that similar results hold with respect to the $p$-Schatten norms \cite{lubotzky2020non}. It is natural, and has already been suggested in \cite{chapmanLII2023stability} and \cite{gohla2024high}, to check whether the \(\widetilde{\Gamma}\) are also non-sofic. In order to show that they are (also) non-sofic, it suffices: 
   
   (a) To prove that the permutation Cheeger constant of the simplicial complex underlying $\Gamma$ is positive, generalizing \cite{EvraK2016}. This would imply  that $\Gamma$ is stable.
   
   (b) To prove that the (flexible) stability of $\Gamma$ implies the non-soficity of $\widetilde{\Gamma}$. 
\\

   Clause (b) was proved by Gohla and Thom \cite{gohla2024high}. Here we offer a more algebraic/combinatorial treatment to their theorem.
\end{abstract}

\iffalse
A long standing problem asks whether every group is sofic, i.e., can be separated by almost-homomorphisms to the symmetric group $\textup{Sym}(n)$. Similar problems have been asked with respect to almost-homomorphisms to the unitary group $U(n)$, equipped with various norms. One of these problems has been solved for the first time in [De Chiffre, Gelbsky, Lubotzky, Thom, 2020]: some central extensions $\widetilde{\Gamma}$ of arithmetic lattices $\Gamma$ of $Sp(2g,\mathbb{Q}_p)$ were shown to be non-Frobenius approximated by almost homomorphisms to $U(n)$. Right after, it was shown that similar results hold with respect to the $p$-Schatten norms in [Lubotzky, Oppenheim, 2020]. It is natural, and has already been suggested in [Chapman, Lubotzky, 2024] and [Gohla, Thom, 2024], to check whether the $\widetilde{\Gamma}$ are also non-sofic. In order to show that they are (also) non-sofic, it suffices: 
   
   (a) To prove that the permutation Cheeger constant of the simplicial complex underlying $\Gamma$ is positive, generalizing [Evra, Kaufman, 2016]. This would imply  that $\Gamma$ is stable.
   
   (b) To prove that the (flexible) stability of $\Gamma$ implies the non-soficity of $\widetilde{\Gamma}$. 

Clause (b) was proved by Gohla and Thom. Here we offer a more algebraic/combinatorial treatment to their theorem.
\fi

\def\G{G}
\def\tG{\tilde{G}}
\def\k{\tau}
\def\g{g}
\def\Id{\textup{Id}}
\def \FF{\mathbb{F}}

\def \Img{\textup{Im}}
\def \sign{\textup{sign}}
\def \defect{\textup{def}}

\def \Sym{\textup{Sym}}
\section{Introduction}
A group is said to be \emph{sofic} if its elements can be separated by almost actions on finite sets.
This notion is due to Gromov \cite{MR1694588} and Weiss  \cite{weiss-2000}, who were motivated by Gottshalk's Conjecture. The class of sofic groups includes all residually finite groups, as well as all amenable groups, and there are sofic groups that are known to be neither. Deciding whether all groups are sofic or not remains a major open problem.
Separation properties by almost homomorphisms to other ``nice'' metric groups are studied, and are known as \emph{approximation properties}. 
Specifically, the case of almost unitary representations, where the unitaries are equipped with some matrix norm, is well studied. For example, Connes \cite{connes1976classification} studied hyperlinearity,  Brown and Ozawa \cite{brown2008textrm} studied property MF, and De Chiffre--Glebsky--Lubotzky--Thom \cite{de2020stability} studied Frobenius approximable groups. The latter were the first to provide examples of  \emph{non-approximable groups}, namely groups whose elements \textbf{cannot} be separated by almost representations (with respect to the Frobenius norm in their case).\footnote{This technique was generalized to $p$-Schacten non-approximability in \cite{lubotzky2020non}.}  Let us say something about their construction, as it is also the study of this paper.  By imitating a classical work of Deligne in the $\mathbb{R}$-coefficents case \cite{deligne1978extensions}, De Chiffre--Glebsky--Lubotzky--Thom   constructed  lattices $\Gamma$ in $\textrm{Sp}(2g,\mathbb{Q}_p)$ that have  central finite extensions 
\begin{equation}\label{eq:the_Deligne_p-adic_extension}
    1\to Z\to \widetilde{\Gamma}\to\Gamma\to 1,
\end{equation}
where $Z$ is the cyclic group of order $p-1$ and every finite index subgroup of $\widetilde{\Gamma}$ contains the unique index $2$ subgroup of $Z$, which implies in particular that $\widetilde{\Gamma}$ is not residually finite if $p\geq 5$. They proved that these extensions are Frobenius stable, and by a simple observation of Glebsky and Rivera \cite{GlebskyRivera} this implies that  $\widetilde{\Gamma}$ is non-Frobenius approximable. Throughout this paper, we call the $\widetilde{\Gamma}$ from their construction the \emph{$p$-adic Deligne extension}.

In \cite{ChapmanL2023stability,chapmanLII2023stability} the authors provide permutation coefficients analogues of standard cohomological notions --- specifically, cocycle expansion (see \pref{def:cocycle_and_coboundary_exp}) and large cosystols (see \pref{def:kappa_large_cosyst}) of a simplicial complex. These cohomological notions are the backbone to the work of Kaufman--Kazhdan--Lubotzky and Evra and Kaufman \cite{KaufmanKL2014,EvraK2016} who resolved Gromov's topological overlapping problem \cite{Gromov2010}.
This permutation setup led \cite{chapmanLII2023stability}, to suggest a plan to prove that the  $p$-adic Deligne extension $\widetilde{\Gamma}$ \eqref{eq:the_Deligne_p-adic_extension} is non-sofic. A similar plan was suggested in \cite{gohla2024high}. The plan consists of two main ingredients:
\begin{enumerate}[(1)]
    \item \label{clause:first_part_plan} The above \(\Gamma\), being a cocompact lattice in \(\textrm{Sp}(2g,\mathbb{Q}_p)\), acts on the Bruhat-tits building \(\mathcal{B}(\textrm{Sp}(2g,\mathbb{Q}_p)\)). Assuming that \(\Gamma\) is torsion free (which we may by passing to a finite index subgroup), it is the fundamental group of \(X = \Gamma \setminus \mathcal{B}(\textrm{Sp}(2g,\mathbb{Q}_p)\). The simplicial complex $X$ was shown in \cite{KaufmanKL2014,EvraK2016}  to be a cocycle expander with $\FF_2$-coefficients. Cocycle expansion of $X$ with \textbf{permutation coefficients} would, in particular, imply that $\Gamma$ is stable (see \pref{def:rho_stability}).
    \item \label{clause:second_part_plan} Conditional on $\Gamma$ being stable, Gohla and Thom \cite{gohla2024high} proved that  $\widetilde{\Gamma}$ is non-sofic.
\end{enumerate}

It is very tempting to try to prove that in a situation like \eqref{eq:the_Deligne_p-adic_extension}, if $\Gamma$ is stable, then also its finite central extension $\widetilde{\Gamma}$. This would imply immediately that $\widetilde{\Gamma}$ is non-sofic, as sofic stable groups are residually finite \cite{GlebskyRivera}. In \pref{sec:non-stale_central_extension}, we will show that this does not hold in general: We present an exact sequence similar to \eqref{eq:the_Deligne_p-adic_extension}, where $\Gamma$ is replaced by some other group (the lamplighter group), which is stable \cite{levit2022infinitely}, but the central extension there is non-stable. The example there is finitely generated but infinitely presented --- see \pref{sec:non-stale_central_extension}. We do not know if this is possible with finitely presented groups.
\\

The main goal of this note is to offer a combinatorial proof to the above ingredient (\ref{clause:second_part_plan}) along the lines of \cite{ChapmanL2023stability,chapmanLII2023stability}. We will show that this can be deduced fairly directly from the ``large cosystol condition'' over \(\F_2\) of \(X\) and its finite covers, that was proved in \cite{KaufmanKL2014,EvraK2016} in order to deduce topological overlapping.

\subsubsection*{Stability, Soficity and Cohomology with $\FF_2$-coefficients}

\begin{definition}[Hamming distance with errors]\label{def:Hamming_metric_errors}
    Let $\Omega\subseteq \Sigma$ be finite sets. The (normalized) \emph{Hamming distance} (with errors) between $\sigma\in \Sym(\Omega)$ and $\sigma'\in \Sym(\Sigma)$ is
\[
d_H(\sigma,\sigma'):=1-\frac{|\{\star\in \Omega\mid \sigma.\star=\sigma'.\star\}|}{|\Sigma|}=\Pro{\star\in \Sigma}[\sigma.\star\neq \sigma'.\star],
\]
where $\sigma.\star=\mathfrak{error}\notin \Sigma$ for every $\star\notin \Omega$, and $\Pro{\star\in \Sigma}$ is the uniform distribution over $\Sigma$.\footnote{Though it is not clear that $d_H$ satisfies the triangle inequality, it does (see \cite{becker2022stability}).} This distance extends to (an $L^\infty$-type distance on) the functions $\varphi\colon S\to \Sym(\Omega)$ and $\psi\colon S\to \Sym(\Sigma)$ by letting 
\[
d_H(\varphi,\psi):=\max_{s\in S}[d_H(\varphi(s),\psi(s))].
\]
\end{definition}
\begin{definition}[Defect and Almost Actions]\label{def:Defect_Almost_Actions}
Let $\Gamma=\langle S|R\rangle$ be a finitely presented group.
The ($L^{\infty}$-)\emph{defect} of a function $\varphi\colon S\to \Sym(\Omega)$ (with respect to $\langle S|R\rangle$) is defined to be 
\[
    \defect(\varphi):=\max_{r\in R}[d_H(\varphi(r),\Id_\Omega)]= \max_{r\in R}\Pro{\star\in \Omega}[\varphi(r).\star\neq \star].
\]
We say that $\varphi$ is an \emph{$\eps$-almost action} of $\Gamma$ on $\Omega$ if $\defect(\varphi)\leq\eps$.\footnote{Note that in \cite{ChapmanL2023stability,chapmanLII2023stability} the $L^1$-defect and $L^1$-distance between functions is used. In the context of this paper it makes the calculations nicer to use the $L^\infty$-setup, so we chose it.}
\end{definition}

\begin{definition}[Soficity]\label{def:soficity}
    A  finitely presented group $\Gamma=\langle S|R\rangle$   is \emph{sofic} if there is a family of $\eps_n$-almost actions $\varphi_n\colon S\to \Sym(\Omega_n)$, where $\eps_n\xrightarrow{n\to \infty}0$, and for every $w\in \mathcal{F}_S$ which is not in the normal closure $\langle\langle R\rangle\rangle$, $d_H(\varphi_n(w),\Id_{\Omega_n})\xrightarrow{n\to\infty}1$.
Such a family is called a \emph{sofic approximation}.
\end{definition}

\begin{definition}[$\rho$-stability]\label{def:rho_stability}
    A function $\rho\colon [0,1]\to [0,1]$ is said to be a \emph{rate function} if $\rho(\eps)\xrightarrow{\eps\to 0}0$. A finitely presented group $\Gamma=\langle S|R\rangle$ is said to be $\rho$-stable,\footnote{This is commonly called pointwise flexible stability in permutations (see the discussions in \cite{becker2022stability}).} for some rate function $\rho$, if for every $\eps$-almost action $\varphi\colon S\to \Sym(\Omega)$, there is a finite superset $\Omega\subseteq \Sigma$ and a  $\Gamma$-action $\psi\colon S\to \Sym(\Sigma)$ such that $d_H(\varphi,\psi)\leq \rho(\eps)$.
\end{definition}

We can now formally state the corollary from the work of Gohla and Thom that is important for us:
\begin{theorem}[Gohla--Thom \cite{gohla2024high}]\label{thm:GT}
    For every dimension $g$, there exists a prime number $p_g$, such that the following holds for all primes \(p\) larger than $p_g$. If $\Gamma$ from the $p$-adic  Deligne extension \eqref{eq:the_Deligne_p-adic_extension} is  $\rho$-stable (\pref{def:rho_stability}), then the extension  $\widetilde{\Gamma}$ is non-sofic (\pref{def:soficity}). 
\end{theorem}

\begin{definition}[Simplicial complexes and $\FF_2$-cohomology]\label{def:simp_cpxs_and_cohomology}
    Let $X$ be a finite, $d$-dimensional, pure simplicial complex.  Namely: there is a finite set $X(0)$; a $k$-cell $\sigma$ in $X$ is a subset of $X(0)$ of size $k+1$, i.e., $\sigma\subseteq X(0)$ and $|\sigma|=k+1$; if $\sigma\in X$ and $\tau\subseteq \sigma$, then $\tau\in X$; every cell in $X$ is contained in some $d$-cell.  A $k$-cochain with $\FF_2$-coefficients is a function $\alpha\colon X(k)\to \FF_2$, and the collection of $k$-cochains is denoted by $C^k=C^k(X;\FF_2)$. The coboundary map $\delta=\delta_k\colon C^k\to C^{k+1}$ is defined by 
\[
\forall \sigma\in X(k+1)\ \colon \ \ \delta\alpha(\sigma)=\sum_{\substack{\tau\subseteq\sigma\\ \tau\in X(k)}}\alpha(\tau).
\]
The $k$-cocycles $Z^k=Z^k(X;\FF_2)$ is the kernel of $\delta_k$ while the $k$-coboundaries $B^k=B^k(X;\FF_2)$ is the image of $\delta_{k-1}$. It is straightforward to check that $\delta_{k}\circ\delta_{k-1}=0$, which means $B^k$ is a subspace of $Z^k$.  The $k^\textrm{th}$ cohomology group of $X$ with $\FF_2$ coefficients $H^k(X;\FF_2)$ is the quotient $\nicefrac{Z^k}{B^k}$.

There is a fully supported probability density function $w$ on $k$-cells of $X$, defined by 
\begin{equation}\label{eq:def_of_weight_of_cell}
    \forall \sigma \in X(k)\ \colon \ \ w(\sigma)=\frac{1}{\binom{d+1}{k+1}}\Pro{\tau\in X(d)}[\sigma\subseteq\tau],
\end{equation}
where $\Pro{\tau\in X(d)}$ is the uniform distribution over $d$-cells in $X$.
This weight function induces a norm  on cochains as follows
\begin{equation}\label{eq:weighted_norm_on_cells}
    \forall \alpha\in C^k\ \colon \ \ \|\alpha\|=\sum_{\sigma\colon \alpha(\sigma)=1}w(\sigma),
\end{equation}
which in turn induces a metric $d_w$ on the $k$-cochains by letting $d_w(\alpha_1,\alpha_2)=\|\alpha_1-\alpha_2\|$. This is also the probability \(\alpha_1(\sigma) \ne \alpha_2(\sigma)\) when sampling \(\sigma\) by sampling a \(d\)-face \(\rho \in X(d)\) uniformly at random, and then taking \(\sigma \subseteq \rho\) uniformly at random. We denote the probability defined by this distribution by \(\Pr_{\sigma \overset{w}{\sim} X(i)}\), that is, for every set $E \subseteq X(i)$, \[\Prob[\sigma \overset{w}{\sim} X(i)]{E} = \sum_{\sigma \in E}w(\sigma).\]
\end{definition}

\begin{definition}[Cocycle and Coboundary expansion]\label{def:cocycle_and_coboundary_exp}
       A (finite, $d$-dimensional, pure) simplicial complex $X$ is said to be an \emph{$\eta$-cocycle expander} (this $\eta$ is commonly referred to as the Cheeger constant of the complex) for some $\eta>0$, if for every $0\leq k\leq d-1$, and every $k$-cochain $\alpha$, we have
       \[
\|\delta\alpha \|\geq \eta\cdot d_w(\alpha,Z^k),
       \]
       where $d_w(\alpha,Z^k)$ is the distance of $\alpha$ to the closest $k$-cocycle.
       It is said to be an \emph{$\eta$-coboundary expander} if in addition the cohomology groups $H^k(X;\FF_2)$ vanish for all $0\leq k\leq d-1$.
\end{definition}
While the above notion is not used in this part, it is used in \cite{EvraK2016} to deduce the following definition which we do need. In addition, its generalization to permutations, namely cocycle expansion with permutation coefficients, is a strong version of stability of the fundamental group of the complex, as discussed in \cite{ChapmanL2023stability,chapmanLII2023stability}.
\begin{definition}[Large cosystoles]\label{def:kappa_large_cosyst}
    A (finite, $d$-dimensional, pure) simplicial complex $X$ is said to have \emph{$\kappa$-large cosystoles} for some $\kappa>0$,  if for every $0\leq k\leq d-1$, and every $k$-cocycle $\alpha$ which is \textbf{not} a $k$-coboundary, we have 
\[
d_w(\alpha,B^k)\geq \kappa,
\]
where $d_w(\alpha,B^k)$ is the distance of $\alpha$ to the closest $k$-coboundary.
\end{definition}

\begin{theorem}[{\cite[Corollary 6.3, Remark 6.5]{EvraK2016}}] \label{thm:ek16}
    For every positive integer \(g\), and a large enough $p$ (with respect to $g$), there exists \(\kappa=\kappa(p,g) > 0\) such that the following holds. Let \(X\) be a finite quotient of \(\mathcal{B}(\textrm{Sp}(2g,\mathbb{Q}_p))\). Then, the \((d-1)\)-skeleton of \(X\) (and all of its finite covers) have \(\kappa\)-large cosystols.
\end{theorem}

\subsubsection*{The combinatorial statement and proof}
    The notion of coboundary expansion was developed independently by Linial--Meshulam \cite{LinialM2006} and by Gromov \cite{Gromov2010}. Gromov proved that coboundary expansion implies topological overlapping. Later, it was shown in \cite{DotterrerKW2018} that it is enough to be a cocycle expander with large cosystols (a.k.a\ cosystolic expander) to satisfy topological overlapping. In \cite{KaufmanKL2014,EvraK2016} it was shown that the $d$-dimensional skeleton of certain quotients of  $(d+1)$-dimensional Bruhat--Tits buildings are cocycle expanders with large cosystols, providing the first examples of bounded degree complexes with the topological overlapping property. 
    The group $\Gamma$ in the  $p$-adic Deligne extension \eqref{eq:the_Deligne_p-adic_extension} is the fundamental group of a complex $X$ that satisfies the conditions of \cite{KaufmanKL2014,EvraK2016}, and in particular has large cosystols. 
    Actually, because the conditions of \cite{KaufmanKL2014,EvraK2016} are \textbf{local} in nature, the same is true for every (connected) finite covering space of $X$, namely, all these coverings have large cosystols. 
    The main goal of this paper is to provide a combinatorial treatment to the following theorem of Gohla and Thom:
\begin{theorem}[Theorem 3.13 from \cite{gohla2024high}]\label{thm:main}

    Let $X$ be a finite, $3$-dimensional, pure simplicial complex with fundamental group $\Gamma$. Assume that:
    \begin{enumerate}
        \item Every finite cover of $X$ has $\kappa$-large cosystoles (\pref{def:kappa_large_cosyst}) for some constant $\kappa>0$.
        \item The fundamental group of $X$, denoted by  $\Gamma$, 
     is $\rho$-stable (\pref{def:rho_stability}) for some rate function $\rho$.
     \item The group $\Gamma$ has a central $2$-extension 
  $
0\to \nicefrac{\Z}{2\Z}\xrightarrow{\iota}\widetilde{\Gamma}\xrightarrow{\pi}\Gamma\to 1,$ such that \(\widetilde{\Gamma}\) is not residually finite.
    \end{enumerate} 
    Then, $\widetilde{\Gamma}$ is non-sofic.
\end{theorem}
\begin{remark}
    We should mention that the result in \cite{gohla2024high} is stronger than \pref{thm:main}. They use a weaker notion of stability. They also use a general finite central subgroup \(A\) instead of \(\nicefrac{\Z}{2\Z}\), but that part can also be covered by our method by appealing to \cite{KaufmanM2018} (see also \cite{DiksteinD2023}) instead of \cite{EvraK2016}. We prefer to keep this note as elementary and self contained as possible.
\end{remark}

\begin{proof}[Proof of \pref{thm:GT} assuming \pref{thm:main}]
For simplicity, we will assume $p$ is congruent to $1$ modulo $4$.
In this case, 
by moving to finite index subgroups and quotients, we can assume that the center $Z$ of the  $p$-adic Deligne extension \eqref{eq:the_Deligne_p-adic_extension} is the order $2$-group $\nicefrac{\Z}{2\Z}$. We also assume that \(p\) is large enough so that \pref{thm:ek16} applies for a complex \(X\) whose fundamental group is \(\Gamma\). Moreover, if we denote by $\tau$ the non-trivial element of the center, then it is trivial in every finite quotient of $\widetilde{\Gamma}$. As $\Gamma$ in the  $p$-adic Deligne extension \eqref{eq:the_Deligne_p-adic_extension} is the fundamental group of a complex $X$ that satisfies the conditions of \pref{thm:ek16}, there is some $\kappa>0$ such that every connected covering $Y$ of $X$ has $\kappa$-large cosystols. The rest of the conditions match, and hence we are done.
\end{proof}
    
Let us sketch our proof of \pref{thm:main}, as it is quite elementary. It is standard (see  \pref{fact:correspondence_2_ext_and_cohomology}) that a central $2$-extension of the fundamental group of a simplicial complex $X$ defines a $2$-cocycle $\phi$ on $X$ (up to a $2$-coboundary, namely a cohomology class).  If $\xi$ is an almost action of $\widetilde{\Gamma}$ that (almost) separates $\tau=\iota(1)$ from $\Id_{\widetilde{\Gamma}}$, then  it is standard (see Section \ref{sec:stability_results}) that there is a close almost action $\psi$ that maps $\tau$ to  a fixed point free involution that commutes with the $\psi$-images of all other generators of $\widetilde{\Gamma}$. By looking at the action induced by $\psi$ on $\psi(\tau)$-orbits, we get an almost action $\varphi$ of $\Gamma$. 
As $\Gamma$ is stable, there is a $\Gamma$-action $f$ which is close to $\varphi$. 
As $\Gamma$ is the fundamental group of $X$, such a $\Gamma$-action $f$ defines a covering $\mathcal{P}\colon Y \to X$, and the fundamental group $H$ of (any connected component of) $Y$ can be seen as a subgroup of $\Gamma=\pi_1(X)$. 
Moreover, by assumption, such a $Y$ has large cosystols. By the correspondence theorem, there is a finite index subgroup $\widetilde{H}\leq \widetilde{\Gamma}$ such that $\tau\in \widetilde{H}$ and $\pi(\widetilde{H})=H$, and the extension $1\to \nicefrac{\Z}{2\Z}\to \widetilde{H}\to H\to 1$ induces a $2$-cocycle $\phi'$ on $Y$.  
As \(\widetilde{\Gamma}\) is not residually finite, it must be the case that $\tau$ is trivial in every finite quotient of $\widetilde{\Gamma}$. Hence this extension does not split, which in particular implies $\phi'$ is a non-coboundary cocycle, and this means $\phi'$ must be far from any coboundary. 
On the other hand, it is not hard to use $\psi$ to define a $1$-cochain $\zeta$ on $Y$ such that  the coboundary 
 $\delta \zeta$  is close to $\phi'$ (see \eqref{eq:1}). This is a contradiction to the existence of the original $\xi$, namely, $\tau$ \textbf{cannot} be separated from $\Id_{\widetilde{\Gamma}}$ by almost actions, and thus $\widetilde{\Gamma}$ is non-sofic.

\subsection*{Acknowledgements}
We thank Corentin Bodart for providing us \pref{claim:separation_of_z_implies_periodicity} and \pref{cor:Gamma_I_not_res_finite} below. Michael Chapman acknowledges with gratitude the Simons Society of Fellows and is supported by a grant from the Simons Foundation (N. 965535). 
Yotam Dikstein kindly acknowledges the support by the National Science Foundation under Grant No. DMS-1926686.
Alexander Lubotzky acknwoledges with gratitude the support of the European Research Council (ERC)
under the European Union's Horizon 2020 (N. 882751), and the research grant from the Center for New Scientists at the Weizmann Institute of Science.

\section{Basic observations on stability and cohomology}
This Section contains  standard definitions and results. We include  proofs for a self contained and clear presentation.

\subsection{Elementary stability results}\label{sec:stability_results}

\begin{claim}[Close to almost action is an almost action]\label{claim:perturbing_almost_actions}
    Let $\Gamma=\langle S|R\rangle$ be a finitely presented group, and let $\ell$ be the maximal length of a word in $R$. Let  $\varphi\colon S\to \Sym(\Omega)$  and $\psi\colon S\to \Sym(\Sigma)$, where $\Omega\subseteq\Sigma$, be such that $\varphi$ is an $\eps_1$-almost action of $\Gamma$, and $d_H(\varphi,\psi)\leq \eps_2$. Then $\psi$ is an $(\eps_1+(\ell+1)\cdot \eps_2)$-almost action of $\Gamma$.
\end{claim}

\begin{proof}
Recall that if $\sigma,\sigma'\in \Sym(\Omega)$ and $\zeta,\zeta'\in \Sym(\Sigma)$, then the following triangle inequality holds:\begin{equation}\label{eq:triangle_ineq_Hamming_metric}
d_H(\sigma\sigma',\zeta\zeta')\leq d_H(\sigma,\zeta)+d_H(\sigma',\zeta').    
\end{equation}
This is because, if $\sigma\sigma'.\star\neq \zeta\zeta'.\star$, then one of the following must occur  --- either $\sigma'.\star\neq \zeta'.\star$ or $\sigma'.\star=\zeta'.\star=\diamond$ but $\sigma.\diamond\neq \zeta.\diamond$. The probability of the first occurring is  $d_H(\sigma',\zeta')$, and the probability of the second occurring is at most $d_H(\sigma,\zeta)$, which proves the inequality. 

Recall that 
$\max_{s\in S}(d_H(\varphi(s),\psi(s))=d_H(\varphi,\psi)=\eps_2$. 
By  using  \eqref{eq:triangle_ineq_Hamming_metric}, we can deduce that for every word in the free group over \(S\), $w\in \cF_S$ of length at most  $\ell$, $d_H(\varphi(w),\psi(w))\leq \ell \cdot \eps_2$. In particular, for every $r\in R$ we have $d_H(\varphi(r),\psi(r))\leq \ell \cdot \eps_2$. From the triangle inequality of the Hamming distance with errors, for every $r\in R$, 
\[
d_H(\psi(r),\Id_\Sigma)\leq d_H(\psi(r),\varphi(r))+d_H(\varphi(r),\Id_\Omega)+d_H(\Id_\Omega,\Id_\Sigma)\leq \ell\cdot \eps_2+d_H(\varphi(r),\Id_\Omega)+\eps_2.
\]
Note that in the right-most inequality we use the inequality \(d_H(\Id_\Omega,\Id_\Sigma) \leq \eps_2\) that follows the fact that \(d_H(\varphi,\psi) \leq \eps_2\). Thus
\[
\defect(\psi)=\max_{r\in R}[d_H(\psi(r),\Id_\Sigma)]\leq (\ell+1)\eps_2+\underbrace{\max_{r\in R}[d_H(\varphi(r),\Id_\Omega)]}_{\defect(\varphi)}\leq (\ell+1)\eps_2+\eps_1,
\]
as claimed.
\end{proof}

\begin{claim} [Almost involutions are close to involutions]\label{claim:fixing_to_involution}
\label{claim:2-torsion-Sym} Let $\zeta\in\Sym\left(\Omega\right)$. Then,
there is $\tau\in\Sym\left(\Omega\right)$ such that $\tau^{2}=\Id$ and
$d_{H}\left(\zeta,\tau\right)= d_{H}\left(\zeta^{2},\Id\right)$.
\end{claim}
\begin{proof}
Let $W=\left\{ \star\in \Omega \mid\zeta^{2} .\star=\star\right\} $.
Note that the restriction $\zeta|_{W}$ is an involution $W\rightarrow W$.
Define 
\[
\tau\left(\star\right)=\begin{cases}
\zeta\left(\star\right) & \star\in W,\\
\star & \star\notin W.
\end{cases}
\]
Then $\tau^{2}=\Id$ and $d_{H}\left(\zeta,\tau\right)=1-\frac{|W|}{|X|}=d_{H}\left(\zeta^{2},\Id\right)$.
\end{proof}

\begin{claim}[Almost fixed point free involutions are close to fixed point free involutions]\label{claim:fixing_fixed_point_free}
        Let $\zeta\in \Sym(\Omega)$ be an involution. Assume $d_H(\zeta,\Id)\geq 1-\eps$. Then, there is an involution with no fixed points $\tau\in \Sym(\Sigma)$, where $|\Sigma|=2\cdot\lceil\nicefrac{|\Omega|}{2}\rceil$, such that $d_H(\zeta,\tau)\leq 2\eps$.
\end{claim}
\begin{proof}
    Let $W=\{\star\in \Omega\mid \zeta.\star=\star\}$. Since $d_H(\zeta,\Id)=1-\nicefrac{|W|}{|\Omega|}$, we can deduce that $|W|\leq \eps|\Omega|$. If $W$ is odd, add a vertex to it and make it even. Now that $W$ is even sized, we can choose any perfect matching on it and define $\tau$ to be the involution induced by this perfect matching. Extend $\tau$ to the rest of $\Omega$ by letting it act as $\zeta$ on the vertices out of $W$. Then the resulting $\tau$ is an involution with no fixed points, and 
    \[
d_H(\zeta,\tau)\leq \frac{|W|+1}{|\Omega|}\leq 2\eps.
    \]
\end{proof}
Let $\Omega$ be a finite set, $\Omega_\pm=\{\pm\}\times \Omega=\{+\star,-\star\mid \star\in \Omega\}$, and let $\sign\colon \Omega_\pm \to \{\pm\}$ be the projection on the left coordinate, namely  $\sign(\pm\star)=\pm$, while $|\cdot|\colon \Omega_\pm \to \Omega$ is the projection on the right coordinate, namely $|\pm \star|=\star$. Note that for $\spadesuit\in \Omega_\pm$, we have $\spadesuit=\sign(\spadesuit)|\spadesuit|$.

\begin{claim}[Almost commuting with a fixed point free involution is close to commuting with the involution]\label{claim:fixing_involutions_to_commute}
    Let $\Omega$ be a finite set, and let $\Omega_\pm$ be as above. For every permutation $\zeta\in \Sym(\Omega_\pm)$, there is a permutation $\sigma\in \Sym(\Omega_\pm)$ that commutes with sign flips, namely $\sigma(-\Id)=(-\Id)\sigma$, and 
    \[
d_H(\sigma,\zeta)\leq d_H([-\Id,\zeta],\Id), 
    \]
    where $[a,b]=aba^{-1}b^{-1}$ is the multiplicative commutator.
    \end{claim}
    \begin{proof}

  Let $W=\{\star\in \Omega \mid \zeta.-\star=-(\zeta.+\star)\}$. Then $d_H([-\Id,\zeta],\Id)=1-\nicefrac{|W|}{|\Omega|}$. Let $\overline{\sigma}\colon W\to \Omega$  be the function $\overline{\sigma}(\star)=|\zeta.+\star|$. Then, $\overline{\sigma}$ is a partial permutation, namely it is injective. This is because $\overline{\sigma}(\star)=\overline{\sigma}(\diamond)$ for $\star,\diamond\in W$ implies $|\zeta.+\star|=|\zeta.+\diamond|$, and this means \[
  \zeta.+\star=\zeta.+\diamond \quad\textrm{or}\quad \zeta.+\star=-\zeta.+\diamond=_{\diamond\in W} \zeta.-\diamond.
  \]
  As  $\zeta$ is invertible, the two cases are 
  \[
  +\star=+\diamond \quad\textrm{or}\quad +\star=-\diamond,
  \]
  but the second case is impossible as their left coordinates (i.e., their signs) disagree. This means $+\star=+\diamond$,
  and in particular their right coordinates agree, namely $\star=\diamond$, proving $\overline{\sigma}$ is injective. 
  
  As every partial permutation can be extended to a permutation, extend $\overline{\sigma}$ so that it becomes a permutation on $\Omega$. Define $\sigma\colon \Omega_\pm\to \Omega_\pm$ as follows:
  For every $\spadesuit\in \Omega_\pm$, if $|\spadesuit|\in W$, then let $\sigma.\spadesuit=\zeta.\spadesuit$. Otherwise, let $\sigma.\spadesuit=\sign(\spadesuit)\overline{\sigma}.|\spadesuit|$.
It is straightforward that $d_H(\sigma,\zeta)\leq 1-\nicefrac{|W|}{|\Omega|}=d_H([-\Id,\zeta],\Id)$. Moreover, if $\star\in W$, then 
\[
\sigma.-\star=\zeta.-\star=_{\star\in W}-\zeta.+\star=_{\star\in W}-\sigma.+\star.
\]
Otherwise, 
\[
\sigma.-\star=-\overline{\sigma}.\star=-\sigma.+\star,
\]
which implies that $\sigma(-\Id)=(-\Id)\sigma$, as required.
 \end{proof}

\begin{corollary}\label{cor:nice_sofic_approx}
    Let $\Gamma=\langle S|R\rangle$ be a finitely presented sofic group with a generator $\tau\in S$ that is a central non-trivial involution in $\Gamma$. Then, there is a sofic approximation $\psi_n\colon S\to \textrm{Sym}((\Omega_n)_\pm)$, such that $\tau$ is always sent to the sign flip $\psi_n(\tau)=-\Id$, and every permutation in the image $\Img(\psi_n)$ commutes with the sign flip.
\end{corollary}

\begin{proof}
    First, we can assume without loss of generality that $\tau^2\in R$, and that $[\tau,s]\in R$ for every $s\in S$. Furthermore, let $\ell$ be the maximal length of a relation $r\in R$. As $\Gamma$ is sofic, it has some sofic approximation consisting of $\eps_n$-almost actions $\varphi_n\colon S\to \Sym(\Omega_n)$, where $\eps_n\xrightarrow{n\to\infty}0$. In particular, 
    \[
   d_H(\varphi_n(\tau^2),\Id_{\Omega_n})\leq  \eps_n.
    \]
    Let $\eps'_n:=1-d_H(\varphi_n(\tau),\Id_{\Omega_n})$, and as $\Id_\Gamma\neq\tau\in \Gamma$, $\eps'_n\xrightarrow{n\to \infty}0$ as well.

    By \pref{claim:fixing_to_involution}, there are functions $f_n\colon S\to \Sym(\Omega_n)$ such that $f_n(\tau^2)=\Id$, $f_n(s)=\varphi_n(s)$ for every $\tau\neq s\in S$,  and $d_H(f_n(\tau),\varphi_n(\tau))\leq  \eps_n$. Thus, $d_H(f_n,\varphi_n)\leq  \eps_n$, and by \pref{claim:perturbing_almost_actions}, 
    \[
    \defect(f_n)\leq \eps_n+(\ell+1)\cdot \eps_n=(\ell+2)\eps_n. 
    \]
    In addition, using the triangle inequality of the Hamming metric,
    \[
d_H(f_n(\tau),\Id_{\Omega_n})\geq \underbrace{d_H(\varphi_n(\tau),\Id_{\Omega_n})}_{1-\eps'_n}-\underbrace{d_H(f_n(\tau),\varphi_n(\tau))}_{\leq  \eps_n}\geq 1-\eps'_n-\eps_n.
    \]
    Hence, by \pref{claim:fixing_fixed_point_free}, there are functions $g_n\colon S\to (\Sigma_{n})_{\pm}$, where $\Omega_n\subseteq (\Sigma_n)_{\pm}$, such that $g_n(\tau)=-\Id$, and $d_H(g_n,f_n)\leq 2(\eps'_n+\eps_n)$.
    Thus, by \pref{claim:perturbing_almost_actions}, 
    \[
\defect(g_n)\leq \defect(f_n)+(\ell+1)d_H(g_n,f_n)\leq (\ell+2)\eps_n+(\ell+1)\cdot 2(\eps'_n+\eps_n)=(2\ell+2)\eps'_n+(3\ell+4)\eps_n,
    \]
    and in particular
    \[
\forall \tau\neq s\in S\ \colon \ \ d_H(g_n([\tau,s],\Id_{(\Sigma_n)_\pm}))\leq (2\ell+2)\eps'_n+(3\ell+4)\eps_n.
    \]
    Applying \pref{claim:fixing_involutions_to_commute} for every $s\neq \tau$ independently, there are functions $\psi_n\colon S\to \Sym((\Sigma_n)_{\pm})$ such that  $\psi_n(\tau)=-\Id$ (namely $\tau$ is sent to the sign flip), $\psi_n(\tau)\psi_n(s)=\psi_n(s)\psi_n(\tau)$ (namely all the images commute with the sign flip), and $d_H(\psi_n,g_n)\leq (2\ell+2)\eps'_n+(3\ell+4)\eps_n$. Finally, applying \pref{claim:perturbing_almost_actions} once more, we deduce that $\psi_n$ are $O(\eps_n+\eps'_n)$-almost actions (which in particular says $\defect(\psi_n)\xrightarrow{n\to \infty}0$), and for every $w\notin \langle \langle R\rangle \rangle$ we have 
    \[
    \begin{split}
        d_H(\psi_n(w),\Id_{(\Sigma_n)_{\pm}})&\geq d_H(\varphi_n(w),\Id_{\Omega_n})-d_H(\Id_{\Omega_n},\Id_{(\Sigma_n)_\pm})-d_H(\varphi_n(w),f_n(w))\\
        &\ \ -d_H(f_n(w),g_n(w))-d_H(g_n(w),\psi_n(w))\\
        &=1-o_n(1).
    \end{split}
    \]
    Hence, $\psi_n$ is a sofic approximation that satisfies the required conditions.
\end{proof}

\subsection[Cohomology of simplicial complexes and central 2-extensions of their fundamental groups]{Cohomology of simplicial complexes and central $2$-extensions of their fundamental groups} \label{sec:extension-facts}
For the standard definitions of covering spaces and the fundamental group see e.g.\ \cite[Chapter 1]{Hatcher2002}. Here we merely state the facts we need.
Denote by \(\dir{X}(i)\) the set of oriented faces of a complex \(X\).
\begin{fact}[Presenting  the fundamental group of a simplicial complex]\label{fact:presentations_of_fundamental_groups}
    Given a connected simplicial complex $X$ and a rooted spanning tree $(T,v)$ of the $1$-skeleton of $X$. Its fundamental group $\Gamma = \pi_1(X,v)$ has the following presentation \(\Gamma \cong \iprod{S | R}\) where: 
    \begin{enumerate}
        \item The generators are $S=\vec X(1)$,
        \item The relations \(R\) contain each oriented edge $xy$ in the tree $T$, the backtracking relations \(xy \cdot yx\) for every \(xy \in \vec X(1)\), and for each oriented triangle $xyz\in \vec X(2)$ a relation $xy\cdot yz\cdot zx$.
        \item The isomorphism $\Gamma = \pi_1(X,v)$ sends \(xy \in S\) to the loop that starts at the basepoint \(v\), traverses to \(x\) through the tree \(T\), crosses the edge \(xy\), and traverses back from \(y\) to \(v\) again through the tree.
    \end{enumerate}
\end{fact}

\begin{definition}[Central $2$-extensions]\label{def:central_extensions}
    Let $\Gamma$ be a group. A  central $2$-extension of $\Gamma$ is a group $\widetilde{\Gamma}$ together with a short exact sequence
\[
0\to \nicefrac{\mathbb{Z}}{2\mathbb{Z}}\xrightarrow{\iota}\widetilde{\Gamma}\xrightarrow{\pi} \Gamma \to 1,
\]
where $\iota(1)=\tau$ is central in $\widetilde{\Gamma}$.\footnote{The assumption that $\iota(1)$ is central is actually redundant. 
We kept it for emphasis.} 
Two such extensions $\widetilde{\Gamma}_1$ and $\widetilde{\Gamma}_2$ are said to be isomorphic if there is an isomorphism of groups $f\colon \widetilde{\Gamma}_1\to \widetilde{\Gamma}_2$ such that $\iota_2=f\circ \iota_1$ and $\pi_2=f\circ \pi_1$.
\end{definition}

Fixing a complex \(X\) whose fundamental group \(\Gamma\) has a presentation as in \pref{fact:presentations_of_fundamental_groups}, we can describe a correspondence between central \(2\)-extensions and the second cohomology of \(X\). This correspondence is standard but we spell it out to stay self contained. Let \(\widetilde{\Gamma}\) be a central \(2\)-extension and \(\widetilde{S} \subseteq \widetilde{\Gamma}\) be a \(\pi\)-section of the generating set with \(\pi(\widetilde{xy}) = xy\). Note that for every triangle \(xy \cdot yz \cdot zx \in R\), the corresponding \(\widetilde{xy} \cdot \widetilde{yz} \cdot \widetilde{zx} \in \ker \pi = \set{1,\tau}\). Therefore, we can define the following cochain \(\phi^{\widetilde{S}}:X(2) \to \nicefrac{\Z}{2\Z}\) to be such that \(\widetilde{xy}\cdot \widetilde{yz}\cdot \widetilde{zx}=\tau^{\phi^{\widetilde{S}}(xyz)}\). Let us record the following facts, all of which are standard but we prove them to stay self contained.
\begin{fact}\label{fact:correspondence_2_ext_and_cohomology}
With the notation above,
\begin{enumerate}
    \item For every \(\pi\)-section \(\widetilde{S} \subseteq \widetilde{\Gamma}\) of $S$, the function $\phi^{\widetilde{S}}$ is a cocycle, namely \(\phi^{\widetilde{S}} \in Z^2(X,\nicefrac{\Z}{2\Z})\) .
    \item For every two \(\pi\)-sections \(\widetilde{S}_1, \widetilde{S}_2 \subseteq \widetilde{\Gamma}\) of $S$, the difference $\phi^{\widetilde{S}_1} - \phi^{\widetilde{S}_2}$ is a coboundary, namely \(\phi^{\widetilde{S}_1} - \phi^{\widetilde{S}_2}  \in B^2(X,\nicefrac{\Z}{2\Z})\) . Conversely, given a \(\pi\)-section \(\widetilde{S}=\{\widetilde{xy}\mid xy\in \overrightarrow X(1)\} \subseteq \widetilde{\Gamma}\) and a cochain  \( \psi \in C^1(X,\nicefrac{\Z}{2\Z})\) the set $\widetilde{S}'=\{\widetilde{xy}\cdot \tau^{\psi(xy)}\mid xy\in \overrightarrow X(1)\}\subseteq \Gamma$ is a \(\pi\)-section of $S$ and satisfies \( \phi^{\widetilde{S}} - \phi^{\widetilde{S}'}  = \coboundary \psi\).
\end{enumerate}
\end{fact}
\begin{proof}[Proof of \pref{fact:correspondence_2_ext_and_cohomology}]
      Staring with the first item, let \(\phi=\phi_{\widetilde{S}}\) be as in the statement and let $xyzw$ be a $3$-cell of $X$. Then 
\[
\begin{split}
    \Id_{\widetilde{\Gamma}} &=(\widetilde{xy}\cdot\widetilde{yz}\cdot\widetilde{zw}\cdot\widetilde{wy}\cdot\widetilde{yx})\cdot (\widetilde{xy}\cdot \widetilde{yw}\cdot \widetilde{wx})\cdot (\widetilde{xw}\cdot \widetilde{wz}\cdot \widetilde{zx})\cdot (\widetilde{xz}\cdot \widetilde{zy}\cdot \widetilde{yx})\\
    &=\tau^{\phi(yzw)}\cdot\tau^{\phi(xyw)}\cdot\tau^{\phi(xwz)}\cdot\tau^{\phi(xzy)},
\end{split}
\]
where the first equality stems from the fact that this path is trivial in the free group, and the second from the definition of $\phi$ and the fact $\tau$ is central. This implies $\phi(yzw)+\phi(xyw)+\phi(xwz)+\phi(xzy)=0$, i.e.\ that \(\phi\) is a cocycle. 
As for the second item, it is straightforward to check that changing a member \(\widetilde{xy} \in \widetilde{S}\) to \(\tau\cdot  \widetilde{xy}\) corresponds to adding the coboundary \(\coboundary \one_{xy}\) to \(\phi_{\widetilde{S}}\), where  \(\one_{xy}\) is the indicator of the edge \(xy\).
\end{proof}

Given \pref{fact:correspondence_2_ext_and_cohomology}, we can define an injective (yet, not necessarily surjective) correspondence $\Xi$ between the isomorphism classes of central \(2\)-extensions of the fundamental group of a simplicial complex $X$, and the second cohomology group of $X$ with $\nicefrac{\Z}{2\Z}$-coefficients, by  mapping \(\widetilde{\Gamma}\) to the cohomology class \(\Xi(\widetilde{\Gamma}) \coloneqq [\phi^{\widetilde{S}}]\), where $\widetilde{S} \subseteq \widetilde{\Gamma}$ is any \(\pi\)-section of $S=\overrightarrow X(1)$.
\begin{fact}[The correspondence between central $2$-extensions and the second cohomology]\label{fact:correspondence_2_ext_and_cohomology-2}
   Let $\widetilde{\Gamma}_1,\widetilde{\Gamma}_2$ be two central $2$-extensions  of the fundamental group $\Gamma$ of a simplicial complex $X$. If $\Xi(\widetilde{\Gamma}_1)=\Xi(\widetilde{\Gamma}_2)$, then $\widetilde{\Gamma}_1$ and $\widetilde{\Gamma}_2$ are isomorphic. Moreover, the split extension $\nicefrac{\Z}{2\Z}\times \Gamma$ is mapped by $\Xi$ to the cohomology class of the coboundaries, namely $ \Xi(\nicefrac{\Z}{2\Z}\times \Gamma) = B^2(X,\nicefrac{\Z}{2\Z})\in H^2(X,\nicefrac{\Z}{2\Z})$.
\end{fact}
\begin{proof}[Proof sketch of \pref{fact:correspondence_2_ext_and_cohomology-2}]
We note that \(\widetilde{\Gamma}\) can be reconstructed (up to isomorphism) given $[\phi] = \Xi(\widetilde{\Gamma})$. In particular, it is isomotphic to the following presentation: the generating set is $S\cup \{\tau\}$, where \(S\) is the generating set of \(\Gamma\), and the relations are that $\tau$ is of order two and central, while $\widetilde{xy}\cdot \widetilde{yz}\cdot \widetilde{zx}=\tau^{\phi(xyz)}$ for every triangle in $X$. Verifying that this is indeed isomorphic to \(\widetilde{\Gamma}\) is straightforward. This in particular implies that \(\Xi\) is injective.
Finally, for the split extension, by choosing the $\pi$-section of $\nicefrac{\Z}{2\Z}\times \Gamma$ to be such that there is a  $0$ in the left coordinate of all the representatives, one gets that $\Xi(\widetilde{\Gamma})$ is the zero $2$-cochain.
\end{proof}

Finally, the following fact relates the extension of a subgroup the push forward of the corresponding covering map.
\begin{fact}\label{fact:subgp_extension_cocycle}
    Let \(H \leq \Gamma\) be a finite index subgroup and let $\mathcal{P}\colon Y\to X$ be a combinatorial covering map such that \(H\) is the fundamental group of the simplicial complex \(Y\). Let \(\widetilde{\Gamma}\) be a central \(2\)-extension and let \([\phi] = \Xi(\widetilde{\Gamma})\). Let \(\pi^{-1}(H) = \widetilde{H}\) be with the short exact sequence
    \[
    0\to \nicefrac{\Z}{2\Z}\xrightarrow{\iota}\widetilde{H}\xrightarrow{\pi|_{\widetilde{H}}}H\to 1,
    \]
    and let \([\psi] = \Xi(\widetilde{H})\). Then \([\psi]=[\phi \circ \mathcal{P}]\).
\end{fact}

\begin{proof}[Proof of \pref{fact:subgp_extension_cocycle}]
    By \pref{fact:presentations_of_fundamental_groups} and \pref{fact:correspondence_2_ext_and_cohomology}, \(\Xi(\widetilde{H})\) is independent of the choice of rooted tree for \(Y\) with which we define the presentation of \(H\), and of the \(\pi\)-section of \(\tilde{H}\). Therefore, given \(\phi \in [\phi]\), we will find an appropriate presentation and section of \(H\) such that its corresponding \(\psi \in [\psi]\) has \(\psi = \phi \circ \mathcal{P}\). 
    
    Towards this end, for let \((T,v)\) the chosen tree for \(X\). Let \(F_0,F_1,\dots,F_m\) be the connected components of \(\mathcal{P}^{-1}(T)\). We use the convention that for every \(x \in X(0)\), the vertex \(x_i \in \mathcal{P}^{-1}(x) \cap F_i\) is the unique preimage of \(x\) in \(F_i\). Thus, we take the preimage \(v_0\) of the root \(v \in X(0)\), and take \((T^*,v_0)\) to be a rooted tree for \(Y\) which contains \(F_0,F_1,\dots,F_m\).

    Recall that that an element \(x_i y_j \in Y(1)\), corresponds to (a homotopy class of) a path that starts by traversing from \(v_0\) to \(x_i\) through \(T^*\), then crossing the edge \(x_i y_j\), and finally traversing back from \(y_j\) to \(v_0\) via \(T^*\). This path is homotopically equivalent (via backtracking relations) to the path that traverses from \(v_0\) to \(v_i\) through \(T^*\), from \(v_i\) to \(x_i\) through \(F_i\), crosses the edge \(x_i y_j\), traverses from \(y_j\) to \(v_j\) through \(F_j\) and finally traverses back to \(v_0\) through \(T^*\). Recall that \(\mathcal{P}\) defines the injection \(H \leq \Gamma\). Denoting \(\widetilde p_i\) the path in \(T^*\) from \(v_0\) to \(v_i\), and \(p_i = \mathcal{P}(p_i)\), one obtains that 
    \[\mathcal{P}(x_i y_j) = p_i \cdot xy \cdot p_j^{-1}.\]

    Now given a \(\pi\)-section \(\widetilde{S} \subseteq \widetilde{\Gamma}\) for the generating set of \(\Gamma\) (i.e.\ the edges in \(X\)), whose coycle we denote by \(\phi\). We extend it to a \(\pi\)-section \(\widetilde{S} \cup \set{\widetilde{p_i}}_{i=0}^m \to \tilde{\Gamma}\), and thus get a \(\pi\)-section of the generating set for \(H\), the edges in \(Y\), which is
    \[\widetilde{x_i y_j} = \widetilde{p_i} \cdot \widetilde{xy} \cdot \widetilde{p_j}^{-1},\]
    where again we use the convention that \(x_i \in F_i\) and \(y_j \in F_j\). Let \(\psi\) be the cocycle that corresponds to this section. We claim that \(\psi = \phi \circ \mathcal{P}\). Indeed, let \(x_i y_j z_k \in Y(2)\) such that \(x_i \in F_i, y_j \in F_j\) and \(z_k \in F_k\).

    Thus,
    \begin{align*}
        \tau^{\psi(x_i y_j z_k)} &= \widetilde{x_i y_j} \cdot \widetilde{y_j z_k} \cdot \widetilde{z_k x_i} \\
        &= (\widetilde{p_i} \cdot \widetilde{xy} \cdot \widetilde{p_j}^{-1}) \cdot  (\widetilde{p_j} \cdot \widetilde{yz} \cdot \widetilde{p_k}^{-1}) \cdot  (\widetilde{p_k} \cdot \widetilde{zx} \cdot \widetilde{p_i}^{-1}) \\
        &= \widetilde{p_i} \cdot \widetilde{xy} \cdot \widetilde{yz} \cdot \widetilde{zx} \cdot \widetilde{p_i}^{-1} \\
        &= \widetilde{p_i} \cdot \tau^{\phi(xyz)} \cdot \widetilde{p_i}^{-1} = \tau^{\phi(\mathcal{P}(x_i y_j z_k))}.
    \end{align*}
    The last inequality is because the extension is central. We conclude that the two cocylces are equal.
\end{proof}

\section{Proof of \pref{thm:main}}

    Let \(\Gamma = \iprod{S | R}\) be a presentation of \(\Gamma\) as in \pref{fact:presentations_of_fundamental_groups} with its extension \(\widetilde{\Gamma}\). Let \(\phi=\phi^{\widetilde{S}}\) be as in \pref{fact:correspondence_2_ext_and_cohomology}, for some \(\pi\)-section \(\widetilde{S} \subseteq \widetilde{\Gamma}\). Choose the presentation  $\widetilde{\Gamma}=\langle \widetilde{S}\cup \{\tau\}|\widetilde{R}\rangle$ described in the proof sketch of \pref{fact:correspondence_2_ext_and_cohomology-2}. I.e.\ it is generated by $\widetilde{S}\cup \{\tau\}$, with the relations  $\tau^2=\Id$, $[\tau,\widetilde{xy}]=\Id, 
 \ \widetilde{xy}=\Id$ whenever $xy\in T\subseteq \vec X(1)$, $\widetilde{xy}\cdot\widetilde{yx}=\Id$, and  \(\widetilde{xy} \cdot  \widetilde{yz} \cdot \widetilde{zw} \cdot \tau^{\phi(xyz)} = \Id\) for every triangle $xyz\in X(2)$.
    
    Assume towards contradiction that \(\widetilde{\Gamma}\) is sofic, thus by \pref{cor:nice_sofic_approx} for every \(\eps > 0\) there exists an $\eps$-almost action $\psi\colon \widetilde{S}\cup \{\tau\}\to \Sym(\Omega_\pm)$, where  $\varphi(\tau)=-\Id$, and all other images commute with the sign flip. Let \(\psi\) be such an \(\eps\)-almost action. Define $\varphi\colon S\to \Sym(\Omega)$ by letting $\varphi(xy).\star=|\psi(\widetilde{xy}).\pm\star|$. Hence, for every triangle $xyz\in X(2)$, as all the images of $\psi$ commute with the sign flip, we have
\begin{equation}\label{eq:phi_vs_psi}
\varphi(xy)\varphi(yz)\varphi(zx).\star=|\psi(\widetilde{xy})\psi(\widetilde{yz})\psi(\widetilde{zx}).\pm\star|.
\end{equation}
As $\defect(\psi)\leq\eps$, we have that 
\begin{equation}\label{eq:triangle_almost_ok_psi}
    d_H(\psi(\widetilde{xy})\psi(\widetilde{yz})\psi(\widetilde{zx}),(-\Id)^{\phi(xyz)})\leq  \eps
\end{equation} 
for every triangle $xyz$. Combined with \eqref{eq:phi_vs_psi}, we can deduce that $d_H(\varphi(xy)\varphi(yz)\varphi(zx),\Id_\Omega)\leq \eps,$ which implies $\defect(\varphi)\leq  \eps$. 
As $\Gamma$ is $\rho$-stable, this means that there is a $\Gamma$-action $f\colon S\to \Sym(\Sigma)$, where $\Omega\subseteq \Sigma$, and $d_H(f,\varphi)\leq \rho(\eps)$.  
Let $Y$ be the covering of $X$ induced by $f$, namely it satisfies $Y(0)=X(0)\times \Sigma$, and $(x,\star)(y,\diamond)\in \vec Y(1)$ if and only if $xy\in \vec X(1)$ and $\psi(xy).\diamond=\star$ --- the higher dimensional cells are added by standard algebraic topology arguments \cite{Surowski1984}. Then, $\phi$ induces a $2$-cocycle $\phi'$ on $Y$ by letting $\phi'((x,\star)(y,\diamond)(z,\circ))=\phi(xyz)$.

Let us define a $1$-cochain with $\nicefrac{\Z}{2\Z}$-coefficients on $Y$ as follows: if $(x,\star)(y,\diamond)$ is an edge in $Y$ such that 
\begin{equation}\label{eq:1}
    f(xy).\diamond=\varphi(xy).\diamond=|\psi(\widetilde{xy}).(\pm\diamond)|,
\end{equation} 
then we set \(\zeta((x,\star)(y,\diamond))=0\) if \(\sign(\psi(\widetilde{xy}).(+\diamond)) = +\) and \(\zeta((x,\star)(y,\diamond))=1\) if \(\sign(\psi(\widetilde{xy}).(+\diamond)) = -\). We call edges where \eqref{eq:1} holds \emph{an edge of the first type}. The rest of the edges are assigned $0$ (arbitrarily), and are called \emph{edges of the second type}. Let us observe the coboundary of $\zeta$, namely $\coboundary \zeta$, and show that it is close to \(\phi'\). Let \(\triangle = (x,\star)(y,\diamond)(z,\circ) \in Y(2)\) be such that all edges are of the first type and such that \(\star \in \Omega\) and 
\begin{equation} \label{eq:dist-of-psi}
\psi(\widetilde{xy}) \psi(\widetilde{yz}) \psi(\widetilde{zx}). (+\star) = (-1)^{\phi(xyz)}\star.
\end{equation}
We claim that in this case,
\(\phi'(\triangle) = \coboundary \zeta(\triangle)\):

  As $(z,\circ)(x,\star)$ is of the first type, $\psi(\widetilde{zx}).(+\star)=(-1)^{\zeta((z,\circ)(x,\star))}\circ$. As the images of $\psi$ commute sign flips, this implies that
\[
\psi(\widetilde{xy}) \psi(\widetilde{yz}) \psi(\widetilde{zx}). (+\star)=\psi(\widetilde{xy}) \psi(\widetilde{yz}).(-1)^{\zeta((z,\circ)(x,\star))}\circ=(-1)^{\zeta((z,\circ)(x,\star))}\psi(\widetilde{xy}) \psi(\widetilde{yz}).+\circ\ .
\]
Repeating this calculation twice more, for the edges of the first type $(y,\diamond)(z,\circ)$ and $(x,\star)(y,\diamond)$, we deduce that
\[
\psi(\widetilde{xy}) \psi(\widetilde{yz}) \psi(\widetilde{zx}). (+\star)=(-1)^{\zeta((z,\circ)(x,\star))+\zeta((y,\diamond)(z,\circ))+\zeta((x,\star)(y,\diamond))}\star=(-1)^{\delta\zeta(\triangle)}\star\ .
\]
Combining the above equality with \eqref{eq:dist-of-psi}, we deduce that $\phi(xyz)=\delta\zeta(\triangle)$. On the other hand, $\phi'(\triangle)$ was defined to be equal to $\phi(xyz)$, which proves the required equality. 
\\

Let us bound the weighted distance (induced by \eqref{eq:weighted_norm_on_cells}) between the $2$-cocycle $\phi'$ and the $2$-coboundary $\delta\zeta$. 
By the above analysis, $d_w(\phi',\delta\zeta)$ is bounded by the sum of probabilities, when $\triangle=(x,\star)(y,\diamond)(z,\circ)$ is sampled according to its weight \eqref{eq:def_of_weight_of_cell}, of the following two events:
\begin{enumerate}[(Event 1)]
    \item The triangle $\triangle$ contains an edge of the second type.
    \item Equation \eqref{eq:dist-of-psi} is not satisfied by $\triangle$.
\end{enumerate}
As $Y$ is a covering space of $X$,  (oriented) edges in $Y$ are parametrized by pairs $xy\in \vec X(1),\star\in \Sigma$, and (oriented) triangles in $Y$ are parametrized by $xyz\in \vec X(2),\star\in \Sigma$. 
Furthermore, the weight of the covering cell is $\nicefrac{1}{|\Sigma|}$ times the weight of the cell it covers. 
An edge $(x,\star)(y,f(yx).\star)\in \vec Y(1)$ is of the first type if and only if $f(yx).\star=\varphi(yx).\star$. 
So,
\[
\begin{split}
\Pro{(x,\star)(y,\diamond)\overset{w}{\sim} \vec Y(1)}[(x,\star)(y,\diamond)  \text{ is of the first type}]&=\Ex{xy\overset{w}{\sim} \vec X(1)} \left[ \Pro{\star\in \Sigma}\ [(x,\star)(y,f^{-1}(xy).\star)\in \vec Y(1) \text{ is of the first type}] \right ]\\
&\leq \max_{xy\in \vec X(1)} \Pro{\star\in \Sigma}\ [(x,\star)(y,f^{-1}(xy).\star)\in \vec Y(1)\ {\text{ is of the first type}}]\\
&= \max_{xy\in \vec X(1)}d_H(f(xy),\varphi(xy))\\
&=d_H(f,\varphi)\leq \rho(\eps).
\end{split}
\]
Using the above and a union bound, we can deduce that the probability that (Event $1$) occurs is at most $3\rho(\eps)$.

For \eqref{eq:dist-of-psi} to be satisfied by $\triangle=(x,\star)(y,\diamond)(z,\circ)$, first of all $\star$ needs to be in $\Omega$. But, as sampling a triangle in $\vec Y(2)$ requires sampling $xyz\overset{w}{\sim} \vec X(2)$ and $\star\in \Sigma$ uniformly,  $\star\notin \Omega$  happens exactly $1-\frac{|\Omega|}{|\Sigma|}$ of the times. Recall that $1-\frac{|\Omega|}{|\Sigma|}  \leq d_H(f,\phi)\leq \rho(\eps)$. Given that $\star$ was sampled such that $\star\in \Omega$, as $\defect(\psi)\leq \eps$, we have that \eqref{eq:dist-of-psi} is not satisfied with probability  at most $\eps$. Therefore, the probability (Event 2) occurs is upper bounded by \(\rho(\eps)+\eps\).

All in all, $d_w(\phi',\delta\zeta)\leq \eps+4\rho(\eps)$, which tends to zero with $\eps$. Note that this means that there is a connected component $Y_0$ of $Y$, such that the restrictions of $\phi'$ and $\delta\zeta$ to $Y_0$ are still at most $( \eps+4\rho(\eps))$-apart, namely
\[
d_w(\phi'|_{Y_0},\delta\zeta|_{Y_0})\leq \eps+4\rho(\eps).
\]

On the other hand, as $Y_0$ is a connected covering of $X$, its fundamental group $H$ is a subgroup of $\Gamma$, the fundamental group of $X$. By the subgroups correspondence theorem (fourth isomorphism theorem), as $\pi\colon \widetilde{\Gamma}\to \Gamma$ is an epimorphism, there is a subgroup $\widetilde {H}$ of $\widetilde {\Gamma}$ that contains $\tau$ and satisfies $\pi(\widetilde{H})=H$. Furthermore, according to \pref{fact:subgp_extension_cocycle}, $\phi'$ is in the same cohomology class as the $2$-cocycle defined by the central $2$-extension
\[
0\to \nicefrac{\Z}{2\Z}\xrightarrow{\iota}\widetilde {H}\xrightarrow{\pi|_{\widetilde{H}}}H\to 1.
\] 
 As $\tau\in \widetilde{H}$, and $\tau$ is trivial in every finite quotient of $\widetilde{\Gamma}$, this extension cannot split. Thus, $\phi'|_{Y_0}$ is a non-coboundary cocycle, and by the $\kappa$-large cocycles condition on $Y$, 
we deduce that $d_w(\phi'|_{Y_0},B^2(Y_0;\nicefrac{\Z}{2\Z}))\geq \kappa$. But, $\coboundary \zeta\in B^2$, and we can deduce that 
\[
\kappa\leq \eps + 4\rho(\eps).
\]
as \(\rho\) is a rate function, this gives a positive lower bound on \(\eps\) which is a contradiction. Hence, $\widetilde{\Gamma}$ cannot be sofic.
\qed

\section{A non-stable central extension of a stable group}\label{sec:non-stale_central_extension}
The following family of groups were used in \cite{fournier2023no}. Let $I\subseteq \mathbb{Z}_{>0}$ be a subset of positive integers. Let 
\[
\Gamma=\mathbb{Z}_2 \wr \mathbb{Z}\cong\langle a,t \mid a^2=1\ ,\ \forall n\in \mathbb{Z}_{>0}\ \colon \ [a,t^nat^{-n}]=1 \rangle 
\]
be the (standard) lamplighter group. One can define a $2$-central extension of $\Gamma$ using $I$ in the following way (compare this special case to the general machinery discussed in \pref{sec:extension-facts}):
\begin{equation}\label{eq:def_of_Gamma_I}
\Gamma_I=\left\langle a,t,z\ \middle|\ a^2=z^2=[a,z]=[t,z]=1\ ,\ \forall n\in \mathbb{Z}_{>0}\ \colon \ [a,t^nat^{-n}]=\begin{cases}
    1 & n\notin I\\
    z & n\in I
\end{cases} \right\rangle\ ,    
\end{equation}
with the short exact sequence 
\begin{equation}\label{eq:short_exact_seq_Gamma_I}
    1\to \{1,z\}\hookrightarrow \Gamma_I\xrightarrow{\pi} \Gamma\to 1\ ,
\end{equation}
where $\pi$ is defined by mapping $a\mapsto a, t\mapsto t$ and $z\mapsto 1$. The following claim  and corollary were communicated to us by Corentin Bodart.
\begin{claim}\label{claim:separation_of_z_implies_periodicity}
Let $I\subseteq \mathbb{Z}_{>0}$, and $\Gamma_I$ as in \eqref{eq:def_of_Gamma_I}.    If there is a homomorphism $\varphi\colon \Gamma_I\to F$, where $F$ is a finite group, and $\varphi(z)\neq 1$, then $I_\pm=I\sqcup -I=\{\pm n \mid n\in I\}$ is $|F|$-periodic; namely, if $n\in I_\pm$, then $n\pm |F|\in I_\pm$.
\end{claim}
\begin{proof}
    From the elementary identities $g[h,k]g^{-1}=[ghg^{-1},gkg^{-1}]$ and $[h,k]=[k,h]^{-1}$, together with the fact $z=z^{-1}$ in $\Gamma_I$, we can deduce that 
    \[
    [a,t^nat^{-n}]=t^n[t^{-n}at^n,a]t^{-n}=t^n(\underbrace{[a,t^{-n}at^n]}_{z\ \textrm{or}\ 1})^{-1}t^{-n}=[a,t^{-n}at^n].
    \]
    Hence, the identity on $[a,t^{n}at^{-n}]$ in $\Gamma_I$ can be extended as follows:
    \[
\forall n\in \mathbb{Z}\ \colon \ \ [a,t^{n}at^{-n}]=\begin{cases}
    1 & n\notin I_\pm\\
    z & n\in I_\pm
\end{cases}\ .
    \]
    
    Let $f=|F|$. By Lagrange's theorem, $\varphi(t^f)=\varphi(t)^f=1.$ Hence, if $n\in I_{\pm}$, then 
    \[
\varphi ([a,t^{n+f}at^{-n-f}])=\varphi ([a,t^{n}at^{-n}])=\varphi(z)\neq 1,
    \]
    which implies in particular that $[a,t^{n+f}at^{-n-f}]$ is not the identity in $\Gamma_I$. As $[a,t^{m}at^{-m}]$ is either $1$ or $z$ for every $m\in \mathbb{Z}$, we deduce that $[a,t^{n+f}at^{-n-f}]=z$, and thus $n+f\in I$. The argument for $n-f$ is similar. 
\end{proof}

\begin{corollary}\label{cor:Gamma_I_not_res_finite}
    Let $I=\mathbb{Z}_{>0}$. Then, $\Gamma_I$ \eqref{eq:def_of_Gamma_I} is not residually finite. 
\end{corollary}
\begin{proof}

    As $I\pm=\mathbb{Z}\setminus \{0\}$ is not periodic for any positive integer $f$, by \pref{claim:separation_of_z_implies_periodicity}, there is no homomorphism from $\Gamma_I$ to a finite group that separates $z$ from the identity. In particular, $\Gamma_I$ is not residually finite.
\end{proof}

\begin{fact}\label{fact:Gamma_I_is_sofic}
   Let $I\subseteq \mathbb{Z}_{>0}$. Then, the group $\Gamma_I$ \eqref{eq:def_of_Gamma_I} is solvable. 
\end{fact}
\begin{proof}
    Recall that central extensions of solvable groups are solvable, and that the lamplighter group $\Gamma$ is solvable.
\end{proof}

    \pref{def:soficity} and \pref{def:rho_stability} define soficity and stability for a finitely presented group. For the generalized notion of soficity (for any countable group) see, for example, \cite[Definition 3.1]{pestov2008hyperlinear}. For the generalized notion of stability see, for example, \cite{levit2022infinitely}.

\begin{claim}
    Let $I=\mathbb{Z}_{>0}$. Then, $\Gamma_I$ \eqref{eq:def_of_Gamma_I} is not stable. 
\end{claim}
\begin{proof}
In \cite{GlebskyRivera} it was shown that a sofic and stable group must be residually finite. As $\Gamma_I$ is solvable (Fact \ref{fact:Gamma_I_is_sofic}) and thus sofic, yet not residually finite (Corollary \ref{cor:Gamma_I_not_res_finite}), the claim is deduced. 
\end{proof}

So, fixing $I=\mathbb{Z}_{>0}$, the central extension $\Gamma_I$ in the short exact sequence \eqref{eq:short_exact_seq_Gamma_I} is not stable. But, the lamplighter group $\Gamma$ is stable, by \cite[Theorem 1.3]{levit2022infinitely}. Hence, we found the required example. We do not know if such examples exist where \(\Gamma\) (and hence \(\widetilde{\Gamma}\)) are finitely presented.

\printbibliography
\end{document}